\documentclass[12pt, reqno]{amsart}
\usepackage{amsmath, amsthm, amscd, amsfonts, amssymb, graphicx, color}
\usepackage[bookmarksnumbered, colorlinks, plainpages]{hyperref}
\hypersetup{colorlinks=true,linkcolor=red, anchorcolor=green,
citecolor=cyan, urlcolor=red, filecolor=magenta, pdftoolbar=true}
\usepackage{eufrak}
\usepackage{mathrsfs}
\usepackage{yfonts}
\newtheorem{theorem}{Theorem}[section]

\newtheorem{corollary}[theorem]{Corollary}
\theoremstyle{definition}
\newtheorem{definition}[theorem]{Definition}
\newtheorem{example}[theorem]{Example}

\theoremstyle{remark}
\newtheorem{remark}[theorem]{Remark}
\numberwithin{equation}{section}

  \DeclareMathOperator{\spe}{sp}
\begin{document}
\setcounter{page}{1}

\title[On Pompeiu--\v{C}eby\v{s}ev type inequalities]
{On Pompeiu--\v{C}eby\v{s}ev type inequalities for positive linear
	maps of selfadjoint operators in inner product spaces}

\author[M. W.  Alomari]{Mohammad W. Alomari}

\address{Department of Mathematics, Faculty of Science and
Information Technology, Irbid National University, P.O. Box 2600,
Irbid, P.C. 21110, Jordan.}
\email{\textcolor[rgb]{0.00,0.00,0.84}{mwomath@gmail.com}}




\date{Received: xxxxxx; Revised: yyyyyy; Accepted: zzzzzz.}

\begin{abstract}
In this work, generalizations of some inequalities for continuous
$h$-synchronous ($h$-asynchronous) functions of linear bounded selfadjoint 
operators    under positive linear maps in Hilbert spaces are proved.
\end{abstract}

\maketitle

\section{Introduction}

Let $\mathcal{B}\left( H\right) $ be the Banach algebra of all
bounded linear operators defined on a complex Hilbert space
$\left( H;\left\langle \cdot ,\cdot \right\rangle \right)$  with
the identity operator  $1_H$ in $\mathcal{B}\left( H\right) $. Let
$A\in \mathcal{B}\left( H\right) $ be a selfadjoint linear
operator on $\left( H;\left\langle \cdot ,\cdot \right\rangle
\right)$. Let $C\left(\spe\left(A\right)\right)$ be the set of all
continuous functions defined on the spectrum of $A$
$\left(\spe\left(A\right)\right)$ and let $C^*\left(A\right)$ be
the $C^*$-algebra generated by $A$ and the identity operator
$1_H$.

Let us define the map $\mathcal{G}:
C\left(\spe\left(A\right)\right) \to C^*\left(A\right)$ with the
following properties (\cite{TF}, p.3):
\begin{enumerate}
\item $\mathcal{G}\left(\alpha f + \beta g\right) = \alpha
\mathcal{G}\left(f\right)+\beta \mathcal{G}\left(g\right)$, for
all scalars $\alpha, \beta$.

\item $\mathcal{G}\left(fg\right) = \mathcal{G}\left(f\right)
\mathcal{G}\left(g\right)$ and
$\mathcal{G}\left(\overline{f}\right)=\mathcal{G}\left(f\right)^*$;
where $\overline{f}$ denotes to the conjugate of $f$ and
$\mathcal{G}\left(f\right)^*$ denotes to the Hermitian of
$\mathcal{G}\left(f\right)$.

\item $\left\|\mathcal{G}\left(f\right)\right\|=\left\|f \right\|
= \mathop {\sup }\limits_{t \in \spe\left(A\right)} \left|
{f\left( t \right)} \right| $.

\item $\mathcal{G}\left( {f_0 } \right) = 1_H$ and
$\mathcal{G}\left( {f_1 } \right) = A$, where
$f_0\left(t\right)=1$ and $f_1\left(t\right)=t$ for all $t \in
\spe\left(A\right)$.
\end{enumerate}
Accordingly,  we define the continuous functional calculus for a
selfadjoint operator $A$ by
\begin{align*}
f\left(A\right) = \mathcal{G}\left(f\right)  \text{for all} \,f\in
C\left(\spe\left(A\right)\right).
\end{align*}
If both $f$ and $g$ are real valued functions on $\spe(A)$ then
the following important property holds:
\begin{align}
f\left( t \right) \ge g\left( t \right)  \,\text{for all} \, \,t
\in \spe\left( A \right) \,\,\text{implies}\,\, f\left( A \right)
\ge g\left( A \right), \label{eq1.2}
\end{align}
in the operator order of $\mathcal{B}(H)$.

A linear map  is defined to be   $\Phi:\mathcal{B}\left(\mathcal{H} \right)\to \mathcal{B}\left(\mathcal{K} \right)$ which preserves additivity and
homogeneity, i.e.,  $\Phi \left(\lambda_1 A +\lambda_2 B \right)=
\lambda_1\Phi \left( A  \right)+ \lambda_2\Phi \left( B \right)$
for any $\lambda_1,\lambda_2 \in \mathbb{C}$  and $A, B \in \mathcal{B}\left(\mathcal{H} \right)$. The linear map is positive $\Phi:\mathcal{B}\left(\mathcal{H} \right)\to \mathcal{B}\left(\mathcal{K} \right)$ if it preserves the operator order, i.e., if $A\in \mathcal{B}^+\left(\mathcal{H} \right)$ then $\Phi\left(A\right)\in \mathcal{B}^+\left(\mathcal{K} \right)$, and in this case we write
$\textfrak{B} [\mathcal{B}\left(\mathcal{H} \right),\mathcal{B}\left(\mathcal{K} \right)] $. Obviously, a positive linear map $\Phi$ preserves the order relation, namely
$A\le B \Longrightarrow \Phi\left(A\right)\le \Phi\left(B\right)$ and preserves the adjoint operation $\Phi\left(A^*\right)=\Phi\left(A\right)^*$.
Moreover, $\Phi$ is said to be  normalized (unital) if it preserves the identity operator, i.e. $\Phi\left(1_{\mathcal{H}}\right) = 1_{\mathcal{K}}$, in this case we write
$\textfrak{B}_n [\mathcal{B}\left(\mathcal{H}
\right),\mathcal{B}\left(\mathcal{K} \right)] $.

In \cite{MA} and formally in \cite{MA1}, the author  of this paper generalized the concept
of monotonicity   as follows:
\begin{definition}
	A real valued function $f$ defined on $\left[a,b\right]$ is said
	to be increasing (decreasing) with respect to a positive function
	$h:[a,b]\to \mathbb{R}_+$ or simply $h$-increasing
	($h$-decreasing) if and only if
	\begin{align*}
	h\left( x \right)f\left( t \right) - h\left( t \right)f\left( x
	\right) \ge (\le)\,\, 0,
	\end{align*}
	whenever $t \ge  x$  for every $x,t \in [a,b]$. In special case if
	$h(x)=1$ we refer to the original monotonicity. Accordingly, for
	$0<a<b$ we say that $f$ is $t^r$-increasing ($t^r$-decreasing) for
	$r\in \mathbb{R} $ if and only if
	\begin{align*}
	x \le t \Longrightarrow x^r f\left( t \right)  - t^r f\left( x
	\right) \ge (\le)\,\, 0
	\end{align*}
	for every $x,t \in [a,b]$.
\end{definition}
\begin{example}
	Let $0<a<b$ and define $f :\left[a,b\right]\to \mathbb{R}$  given
	by
	\begin{enumerate}
		\item $f(s)=1$, then $f$ is $t^r$-decreasing for all $r>0$ and
		$t^r$-increasing for all $r<0$.
		
		\item $f(s)=s$, then $f$ is $t^r$-decreasing for all $r>1$ and
		$t^r$-increasing for all $r<1$.
		
		\item $f(s)=s^{-1}$, then $f$ is $t^r$-decreasing for all $r>-1$
		and $t^r$-increasing for all $r<-1$.
		
	\end{enumerate}
\end{example}

\begin{remark}
Every $h$-increasing function is increasing. The
	converse need not be true. For more details see \cite{MA1}.
\end{remark}
 
The concept of synchronization  has a wide range of usage in
several areas of mathematics. Simply, two functions
$f,g:\left[a,b\right]\to \mathbb{R}$ are called synchronous
(asynchronous)  if and only if the inequality
\begin{align*}
\left( { f\left( t \right) - f\left( x \right)} \right)\left( {
	g\left( t \right) -  g\left( x\right)} \right) \ge (\le)\,\, 0,
\end{align*}
holds  for all $x,t\in \left[a,b\right]$.

In \cite{MA1}, Alomari generalized the concept of synchronization     of functions of real variables. Indeed, we have
 \begin{definition}
	\label{def2}The real valued functions $f,g:\left[a,b\right]\to
	\mathbb{R}$ are called   synchronous (asynchronous) with respect
	to a non-negative function $h:[a,b]\to \mathbb{R}_+$ or simply
	$h$-synchronous ($h$-asynchronous)  if and only if
	\begin{align}
	\label{h-syn}\left( {h\left( y \right)f\left( x \right) - h\left(
		x \right)f\left( y \right)} \right)\left( {h\left( y
		\right)g\left( x \right) - h\left( x \right)g\left( y \right)}
	\right) \ge (\le)\,\, 0
	\end{align}
	for all $x,y \in \left[a,b\right]$.
	
	In other words if both $f$ and $g$ are either $h$-increasing or
	$h$-decreasing then $$\left( {h\left( y \right)f\left( x \right) -
		h\left( x \right)f\left( y \right)} \right)\left( {h\left( y
		\right)g\left( x \right) - h\left( x \right)g\left( y \right)}
	\right) \ge0.$$ While, if one of the function is $h$-increasing
	and the other is $h$-decreasing then $$\left( {h\left( y
		\right)f\left( x \right) - h\left( x \right)f\left( y \right)}
	\right)\left( {h\left( y \right)g\left( x \right) - h\left( x
		\right)g\left( y \right)} \right) \le0.$$

	In special case if $h(x)=1$ we refer to the original
	synchronization. Accordingly, for $0<a<b$ we say that $f$ and $g$
	are $t^r$-synchronous ($t^r$-asynchronous) for $r\in \mathbb{R} $
	if and only if
	\begin{align*}
	\left(x^r f\left( t \right)  - t^r f\left(
	x \right) \right) \left(x^r g\left( t \right)  - t^r g\left( x
	\right) \right)\ge (\le)\,\, 0
	\end{align*}
	for every $x,t \in [a,b]$.
\end{definition}

\begin{remark}
	In Definition \eqref{def2}, if $f=g$ then $f$ and $g$   are always
	$h$-synchronous regardless of $h$-monotonicity of $f$ (or $g$). In
	other words, a function $f$ is always $h$-synchronous with itself.
\end{remark}

\begin{example}
	Let $0<a<b$ and define $f,g :\left[a,b\right]\to \mathbb{R}$ given
	by
	\begin{enumerate}
		\item $f(s)=1=g(s)$, then $f$ and $g$ are $t^r$-synchronous for
		all $r\in \mathbb{R}$.
		
		\item $f(s)=1$ and $g(s)=s$, then $f$ is $t^r$-synchronous
		for all $r \in \left( { - \infty ,0} \right) \cup \left( {1,\infty } \right)$ and
		$t^r$-asynchronous for all
		$0<r<1$.
		
		\item  $f(s)=1$ and   $g(s)=s^{-1}$, then $f$ is $t^r$-synchronous
		for all $r \in \left( { - \infty ,-1} \right) \cup \left(
		{0,\infty } \right)$ and $t^r$-asynchronous for all $-1<r<0$.
		
		\item $f(s)=s$ and   $g(s)=s^{-1}$, then $f$ is $t^r$-synchronous
		for all $r \in \left( { - \infty ,-1} \right) \cup \left(
		{1,\infty } \right)$ and $t^r$-asynchronous for all $-1<r<1$.
	\end{enumerate}
\end{example}

In \cite{SD1}, Dragomir studied the \v{C}eby\v{s}ev functional
\begin{align}
\label{cebysev} C\left(f,g;A,x\right):= \left\langle {f\left( A
	\right)g\left( A \right)x,x} \right\rangle - \left\langle {
	g\left( A \right)x,x} \right\rangle \left\langle { f\left( A
	\right)x,x} \right\rangle,
\end{align}
for any selfadjoint operator $A\in \mathcal{B}(H)$ and $x\in H$
with $\|x\|=1$.

In  \cite{SD1}, proved  the following result  concerning continuous
synchronous (asynchronous) functions of selfadjoint linear
operators in Hilbert spaces.
\begin{theorem}
\label{thm1.1}Let $A$ be a selfadjoint operator with
$\spe\left(A\right)\subset \left[\gamma,\Gamma\right]$  for some
real numbers $\gamma,\Gamma$ with $\gamma<\Gamma$. If $f,g: \left[
{\gamma,\Gamma} \right]\to \mathbb{R}$ are continuous and
synchronous (asynchronous) on $\left[ {\gamma,\Gamma} \right]$,
then
\begin{align}
\label{eq1.3} \left\langle {f\left( A \right)g\left( A \right)x,x}
\right\rangle \ge (\le) \left\langle { g\left( A \right)x,x}
\right\rangle \left\langle { f\left( A \right)x,x} \right\rangle
\end{align}
for any $x\in H$ with $\|x\|=1$.
\end{theorem}

In \cite{MA1}, Alomari generalized Theorem \ref{thm1.1} for continuous
$h$-synchronous ($h$-asynchronous) functions of selfadjoint linear
operators in Hilbert spaces by introduciing the  Pompeiu--\v{C}eby\v{s}ev functional
such as:
\begin{multline}
\label{P} \mathcal{P}\left(f,g,h;A,x\right):= \left\langle {h^2
	\left( A \right)x,x} \right\rangle \left\langle {f\left( A
	\right)g\left( A \right)x,x} \right\rangle
\\
-\left\langle {h\left( A \right)g\left( A \right)x,x}
\right\rangle \left\langle {h\left( A \right)f\left( A \right)x,x}
\right\rangle
\end{multline}
for $x\in H$ with $\|x\|=1$. This naturally, generalizes the
\v{C}eby\v{s}ev functional \eqref{cebysev}.

Moreover, he proved the following  essential result:
 \begin{theorem}
 	\label{thm1.2}Let $A$ be a selfadjoint operator with
 	$\spe\left(A\right)\subset \left[\gamma,\Gamma\right]$  for some
 	real numbers $\gamma,\Gamma$ with $\gamma<\Gamma$. Let
 	$h:\left[\gamma,\Gamma\right]\to \mathbb{R}_+$ be a non-negative
 	and continuous function. If $f,g:
 	\left[ {\gamma,\Gamma} \right]\to \mathbb{R}$ are continuous and
 	both $f$ and $g$ are $h$-synchronous ($h$-asynchronous) on $\left[
 	{\gamma,\Gamma} \right]$, then
 	\begin{align}
 	\label{eq1.4} \left\langle {h^2 \left( A \right)x,x} \right\rangle
 	\left\langle {f\left( A \right)g\left( A \right)x,x} \right\rangle
 	\ge (\le) \left\langle {h\left( A \right)g\left( A \right)x,x}
 	\right\rangle \left\langle {h\left( A \right)f\left( A \right)x,x}
 	\right\rangle
 	\end{align}
 	for any $x\in H$ with $\|x\|=1$.
 \end{theorem}

For more related results, we refer the reader to \cite{SD2}, \cite{MB} and \cite{MM}.\\

In this work,  some inequalities for continuous $h$-synchronous ($h$-asynchronous)
functions of linear bounded selfadjoint operators under positive linear maps in Hilbert spaces of the
Pompeiu--\v{C}eby\v{s}ev functional \eqref{P} are proved. The
proof Techniques are similar to that ones used in \cite{SD2}.

\section{Main results}

Let us start with the following result regarding the positivity of
$\mathcal{P}\left(f,g,h;A,x\right)$.
\begin{theorem}
\label{thm2.1}Let $A$ be a selfadjoint operator with
$\spe\left(A\right)\subset \left[\gamma,\Gamma\right]$  for some
real numbers $\gamma,\Gamma$ with $\gamma<\Gamma$. Let $\phi,\varphi:\mathcal{B}\left(\mathcal{H} \right)\to \mathcal{B}\left(\mathcal{K} \right)$ be a linear unital maps. Let
$h:\left[\gamma,\Gamma\right]\to \mathbb{R}_+$ be a non-negative
 and continuous function.  If $f,g:
\left[ {\gamma,\Gamma} \right]\to \mathbb{R}$ are continuous and
both $f$ and $g$ are $h$-synchronous ($h$-asynchronous) on $\left[
{\gamma,\Gamma} \right]$, then
 
 \begin{multline}
\label{eq2.1}\left\langle { \phi \left(h^2 \left( B \right)\right)  y,y } \right\rangle  \cdot\left\langle { \varphi \left(f\left( A \right)g\left(
	A \right)\right)x,x } \right\rangle
\\
+  \left\langle { \varphi \left(h^2 \left( A \right)\right)x,x } \right\rangle \cdot  \left\langle { \phi \left(f\left( B \right)g\left( B
	\right)\right) 
	y,y } \right\rangle   
\\
\ge   \left\langle {\varphi \left(h\left( A \right)f\left( A
	\right)\right)x,x  } \right\rangle \cdot     \left\langle {\phi \left(h\left( B \right)g\left( B \right)\right) y,y } \right\rangle 
\\
+  \left\langle { \varphi \left( h\left( A
	\right)g\left( A \right)\right)x,x } \right\rangle\cdot  \left\langle { \phi \left( h\left( B \right)f\left( B
	\right)\right)y,y } \right\rangle   
\end{multline}
for each $x,y\in H$ with $\|x\|=\|y\|=1$.

 \begin{multline}
\label{eq2.1*}\left\langle { \phi \left(h^2 \left( A \right)\right)  y,y } \right\rangle  \cdot\left\langle { \varphi \left(f\left( A \right)g\left(
	A \right)\right)x,x } \right\rangle
\\
+  \left\langle { \varphi \left(h^2 \left( A \right)\right)x,x } \right\rangle \cdot  \left\langle { \phi \left(f\left( A \right)g\left(A
	\right)\right) 
	y,y } \right\rangle   
\\
\ge  (\le)   \left\langle {\varphi \left(h\left( A \right)f\left( A
	\right)\right)x,x  } \right\rangle \cdot     \left\langle {\phi \left(h\left( A \right)g\left( A \right)\right)y,y } \right\rangle 
\\
+  \left\langle { \varphi \left( h\left( A
	\right)g\left( A \right)\right)x,x } \right\rangle\cdot  \left\langle { \phi \left( h\left( A \right)f\left( A
	\right)\right)y,y } \right\rangle   
\end{multline}
for each $x\in H$ with $\|x\ =1$.

\end{theorem}

\begin{proof}
Since $f$ and $g$ are $h$-synchronous then
\begin{align*}
\left( {h\left( s \right)f\left( t \right) - h\left( t
\right)f\left( s \right)} \right)\left( {h\left( s \right)g\left(
t \right) - h\left( t \right)g\left( s \right)} \right)  \ge 0,
\end{align*}
and this is allow us to write
\begin{multline}
h^2\left( s \right)f\left( t \right)g\left( t \right)+h^2\left( t
\right)f\left( s \right)g\left( s \right) 
\\
\ge h\left( s
\right)h\left( t \right)f\left( t \right)g\left( s \right)
+h\left( s \right)h\left( t \right) g\left( t \right)f\left( s
\right) \label{eq2.2}
\end{multline}
for all $t, s\in [a,b]$. We fix $s \in \left[a,b\right]$ and  apply the functional calculus; property \eqref{eq1.2}  
for inequality \eqref{eq2.2} for the operator $A$, then we have
for each $x \in H$ with $\|x\|=1$, that
\begin{multline*}
 h^2 \left( s \right) 1_H \cdot f\left( A \right)g\left(
		A \right) + h^2 \left( A \right)\cdot f\left( s \right)g\left( s
		\right)  1_H 
\\
\ge   h\left( A \right)f\left( A
		\right)\cdot h\left( s \right)g\left( s \right)  1_H  + h\left( A
		\right)g\left( A \right)\cdot h\left( s \right)f\left( s
		\right)  1_H, 
\end{multline*}
and since $\varphi$ is normalized positive linear map we get
\begin{multline*}
h^2 \left( s \right) 1_H \cdot \varphi \left(f\left( A \right)g\left(
A \right)\right) + \varphi \left(h^2 \left( A \right)\right) \cdot f\left( s \right)g\left( s
\right)  1_H 
\\
\ge   \varphi \left(h\left( A \right)f\left( A
\right)\right) \cdot h\left( s \right)g\left( s \right)  1_H  + \varphi \left( h\left( A
\right)g\left( A \right)\right)\cdot h\left( s \right)f\left( s
\right)  1_H, 
\end{multline*}
and this is equivalent to write
\begin{multline}
h^2 \left( s \right) 1_H \cdot\left\langle { \varphi \left(f\left( A \right)g\left(
A \right)\right)x,x } \right\rangle+  \left\langle { \varphi \left(h^2 \left( A \right)\right)x,x } \right\rangle \cdot f\left( s \right)g\left( s
\right)  1_H 
\\
\ge   \left\langle {\varphi \left(h\left( A \right)f\left( A
	\right)\right)x,x  } \right\rangle \cdot h\left( s \right)g\left( s \right)  1_H  +  \left\langle { \varphi \left( h\left( A
	\right)g\left( A \right)\right)x,x } \right\rangle\cdot h\left( s \right)f\left( s
\right)  1_H, \label{eq2.3}
\end{multline}
Applying property \eqref{eq1.2} again for inequality
\eqref{eq2.3} but for the operator $B$, then we have for each $y\in H$ with $\|y\|=1$, that
\begin{multline*}
  h^2 \left( B \right)  \cdot\left\langle { \varphi \left(f\left( A \right)g\left(
	A \right)\right)x,x } \right\rangle+  \left\langle { \varphi \left(h^2 \left( A \right)\right)x,x } \right\rangle \cdot  f\left( B \right)g\left( B
	\right)    
\\
\ge   \left\langle {\varphi \left(h\left( A \right)f\left( A
	\right)\right)x,x  } \right\rangle \cdot    h\left( B \right)g\left( B \right)  +  \left\langle { \varphi \left( h\left( A
	\right)g\left( A \right)\right)x,x } \right\rangle\cdot  h\left( B \right)f\left( B
	\right),  
\end{multline*}
and since $\phi$ is normalized positive linear map we get
\begin{multline*}
\left\langle { \phi \left(h^2 \left( B \right)\right)  y,y } \right\rangle  \cdot\left\langle { \varphi \left(f\left( A \right)g\left(
	A \right)\right)x,x } \right\rangle+  \left\langle { \varphi \left(h^2 \left( A \right)\right)x,x } \right\rangle \cdot  \left\langle { \phi \left(f\left( B \right)g\left( B
	\right)\right) 
	y,y } \right\rangle   
\\
\ge   \left\langle {\varphi \left(h\left( A \right)f\left( A
	\right)\right)x,x  } \right\rangle \cdot     \left\langle {\phi \left(h\left( B \right)g\left( B \right)\right) y,y } \right\rangle +  \left\langle { \varphi \left( h\left( A
	\right)g\left( A \right)\right)x,x } \right\rangle\cdot  \left\langle { \phi \left( h\left( B \right)f\left( B
	\right)\right)y,y } \right\rangle,  
\end{multline*}
 for each $x,y\in H$ with $\|x\|=\|y\|=1$, which gives the required results in \eqref{eq2.1}. To obtain \eqref{eq2.1*} we set $B=A$ in \eqref{eq2.1}. The revers case follows trivially, and this
 completes the proof.
\end{proof}
\begin{corollary}
 Let $A$ be a selfadjoint operator with
	$\spe\left(A\right)\subset \left[\gamma,\Gamma\right]$  for some
	real numbers $\gamma,\Gamma$ with $\gamma<\Gamma$. Let $\phi,\varphi:\mathcal{B}\left(\mathcal{H} \right)\to \mathcal{B}\left(\mathcal{K} \right)$ be a linear unital maps. Let
	$h:\left[\gamma,\Gamma\right]\to \mathbb{R}_+$ be a non-negative
	and continuous function. If $f,g:
	\left[ {\gamma,\Gamma} \right]\to \mathbb{R}$ are continuous and
	both $f$ and $g$ are  synchronous ( asynchronous) on $\left[
	{\gamma,\Gamma} \right]$, then
	
	\begin{multline*}
	 \left\langle { \varphi \left(f\left( A \right)g\left(
		A \right)\right)x,x } \right\rangle+    \left\langle { \phi \left(f\left( B \right)g\left( B
		\right)\right) 
		y,y } \right\rangle   
	\\
	\ge  (\le)   \left\langle {\varphi \left( f\left( A
		\right)\right)x,x  } \right\rangle      \left\langle {\phi \left( g\left( B \right)\right) y,y } \right\rangle +  \left\langle { \varphi \left(g\left( A \right)\right)x,x } \right\rangle   \left\langle { \phi \left( f\left( B
		\right)\right)y,y } \right\rangle   
	\end{multline*}
	for each $x,y\in H$ with $\|x\|=\|y\|=1$. In special case, the following \v{C}eby\v{s}ev inequality for positive linear maps of selfadjoint operator is valid	
	\begin{multline*}
  \left\langle { \varphi \left(f\left( A \right)g\left(
		A \right)\right)x,x } \right\rangle+      \left\langle { \varphi \left(f\left( A \right)g\left(A
		\right)\right) 
		x,x } \right\rangle   
	\\
	\ge (\le)    \left\langle {\varphi \left( f\left( A
		\right)\right)x,x  } \right\rangle      \left\langle {\varphi \left( g\left( A \right)\right) x,x } \right\rangle +  \left\langle { \varphi \left( g\left( A \right)\right)x,x } \right\rangle  \left\langle { \varphi \left(  f\left( A
		\right)\right)x,x } \right\rangle   
	\end{multline*}
	for each $x\in H$ with $\|x\| =1$.
	
\end{corollary}
\begin{proof}
	Setting $h(t)=1$ in both \eqref{eq2.1} and \eqref{eq2.1*}. Also, in \eqref{eq2.1*}  take $\phi=\varphi$, $B=A$ and $y=x$.
\end{proof}

\begin{remark}
Setting  $\phi=\varphi$, $B=A$ and $y=x$ in \eqref{eq2.1}, we get
 \begin{multline*}
\left\langle { \varphi \left(h^2 \left( A \right)\right)  x,x } \right\rangle  \cdot\left\langle { \varphi \left(f\left( A \right)g\left(
	A \right)\right)x,x } \right\rangle
\\
+  \left\langle { \varphi \left(h^2 \left( A \right)\right)x,x } \right\rangle \cdot  \left\langle { \varphi \left(f\left( A \right)g\left(A
	\right)\right) 
	x,x } \right\rangle   
\\
\ge  (\le)   \left\langle {\varphi \left(h\left( A \right)f\left( A
	\right)\right)x,x  } \right\rangle \cdot     \left\langle {\varphi \left(h\left( A \right)g\left( A \right)\right)x,x } \right\rangle 
\\
+  \left\langle { \varphi \left( h\left( A
	\right)g\left( A \right)\right)x,x } \right\rangle\cdot  \left\langle { \varphi \left( h\left( A \right)f\left( A
	\right)\right)x,x } \right\rangle   
\end{multline*}
	for each $x\in H$ with $\|x\| =1$.
\end{remark}
The following generalization of Cauchy-Schwarz inequality holds.
\begin{corollary}
\label{cor1}Let $A$ be a selfadjoint operator with
$\spe\left(A\right)\subset \left[\gamma,\Gamma\right]$  for some
real numbers $\gamma,\Gamma$ with $\gamma<\Gamma$. Let $\phi,\varphi:\mathcal{B}\left(\mathcal{H} \right)\to \mathcal{B}\left(\mathcal{K} \right)$ be a linear unital maps. Let
$h:\left[\gamma,\Gamma\right]\to \mathbb{R}_+$ be a non-negative
 and continuous function. If $f: \left[
{\gamma,\Gamma} \right]\to \mathbb{R}$ is continuous and
$h$-synchronous  on $\left[ {\gamma,\Gamma} \right]$, then
  \begin{multline}
\label{ineq.cor1} \left\langle { \phi \left(h^2 \left( B \right)\right)  y,y } \right\rangle  \cdot\left\langle { \varphi \left(f^2\left( A \right) \right)x,x } \right\rangle+  \left\langle { \varphi \left(h^2 \left( A \right)\right)x,x } \right\rangle \cdot  \left\langle { \phi \left(f^2\left( B \right)\right) 
	y,y } \right\rangle   
\\
\ge   2  \left\langle {\varphi \left(h\left( A \right)f\left( A
	\right)\right)x,x  } \right\rangle \cdot     \left\langle {\phi \left(h\left( B \right)f\left( B \right)\right) y,y } \right\rangle 
\end{multline}
for each $x,y\in H$ with $\|x\|=\|y\|=1$. In particular, we have
\begin{align}
\label{ineq.cor1*} \left\langle { \varphi \left(h^2 \left( A \right)\right)  x,x} \right\rangle  \cdot\left\langle { \varphi \left(f^2\left( A \right) \right)x,x } \right\rangle 
\ge   \left\langle {\varphi \left(h\left( A \right)f\left( A
	\right)\right)x,x  } \right\rangle^2
\end{align}
for each $x \in H$ with $\|x\|=1$.
\end{corollary}
\begin{proof}
Setting $f=g$ in both \eqref{eq2.1} and \eqref{eq2.1*}. Also, in \eqref{eq2.1*} take $\phi=\varphi$, $B=A$ and $y=x$, so that the desired results hold.
\end{proof}

\begin{corollary}
\label{cor2} Let $A$ be a selfadjoint operator with
$\spe\left(A\right)\subset \left[\gamma,\Gamma\right]$  for some
real numbers $\gamma,\Gamma$ with $0<\gamma<\Gamma$. Let $\phi,\varphi:\mathcal{B}\left(\mathcal{H} \right)\to \mathcal{B}\left(\mathcal{K} \right)$ be a linear unital maps. If $f,g:
\left[ {\gamma,\Gamma} \right]\to \mathbb{R}$ are continuous and
$t$-synchronous ($t$-asynchronous) on $\left[ {\gamma,\Gamma}
\right]$, then
 
  \begin{multline}
\left\langle { \phi \left(  B^2\right)  y,y } \right\rangle  \cdot\left\langle { \varphi \left(f\left( A \right)g\left(
	A \right)\right)x,x } \right\rangle+  \left\langle { \varphi \left( A^2 \right)x,x } \right\rangle \cdot  \left\langle { \phi \left(f\left( B \right)g\left( B
	\right)\right) 
	y,y } \right\rangle   
\\
\ge (\le)   \left\langle {\varphi \left(Af\left( A
	\right)\right)x,x  } \right\rangle \cdot     \left\langle {\phi \left(Bg\left( B \right)\right) y,y } \right\rangle 
\\
+  \left\langle { \varphi \left(A g\left( A \right)\right)x,x } \right\rangle\cdot  \left\langle { \phi \left(   B  f\left( B
	\right)\right)y,y } \right\rangle   
\end{multline}
 
for each $x,y \in H$ with $\|x\|=\|y\|=1$.
\end{corollary}
\begin{proof}
Setting $h(t)=t$ in \eqref{eq2.1} we get the desired result.
\end{proof}

Before we state our next remark, we interested to give  the
following example.
\begin{example}
\label{example2.11}
\begin{enumerate}
\item If $f(s)=s^p$ and $g(s)=s^q$ ($s>0$), then $f$ and $g$ are
$t^r$-synchronous for all $p,q>r>0$  and $t^r$-asynchronous
 for all $p> r>q>0$.

\item If $f(s)=s^p$ and $g(s)=\log (s)$ ($s>1$), then  $f$ is
$t^r$-synchronous
 for all $p < r<0$ and $t^r$-asynchronous
 for all $r< p<0$.

\item If  $f(s)=\exp(s)=g(s)$, then $f$ is $t^r$-synchronous for
all for all $r\in \mathbb{R}$.
\end{enumerate}
\end{example}

\begin{remark}
\label{remark1} Using Example
\ref{example2.11} we can observe the following special cases:
 \begin{enumerate}
 \item If $f\left( s \right)=s^p$ and $g\left( s \right)=s^q$ ($s>0$),
then $f$ and $g$ are $t^r$-synchronous for all $p,q > r>0$, so
that we have
\begin{multline*}
 \left\langle {\phi \left(B^{2r}\right)y,y} \right\rangle \left\langle {\varphi \left(A^{p+q}\right)x,x} \right\rangle  + \left\langle {\varphi\left( A^{2r}\right)x,x} \right\rangle \left\langle {\phi \left(B^{p+q}\right)y,y} \right\rangle  \\
  \ge   \left\langle {\varphi \left(B^{q+r}\right)y,y} \right\rangle \left\langle {\phi \left(A^{p+r}\right)x,x} \right\rangle  + \left\langle {\varphi \left(A^{q+r}\right)x,x} \right\rangle \left\langle {\phi\left(B^{p+r}\right) y,y}
  \right\rangle.
\end{multline*}
If $p>r>q>0$, then $f$ and $g$ are $t^r$-asynchronous and thus the
reverse inequality holds.\\

\item If $f\left( s \right)=s^p$   and $g\left( s \right)=\log
 s$ ($s>1$),
 then $f$ and $g$ are $t^r$-synchronous for all $p < r<0$ we have
\begin{multline*}
 \left\langle {\phi \left(B^{2r}\right)y,y} \right\rangle \left\langle {\varphi\left(A^p \log\left( A \right)\right)x,x} \right\rangle  + \left\langle {\varphi\left(A^{2r}\right)x,x} \right\rangle \left\langle {\phi\left(B^p\log\left( B \right)\right)y,y} \right\rangle  \\
  \ge   \left\langle {\varphi\left(B^r\log\left(B \right)\right)y,y} \right\rangle \left\langle {\phi\left(A^{p+r}\right)x,x} \right\rangle  + \left\langle {\varphi\left(A\log\left( A \right)\right)x,x} \right\rangle \left\langle {\phi\left(B^{p+r}\right)y,y}
  \right\rangle.
\end{multline*}
If $r<p<0$, then $f$ and $g$ are $t^r$-asynchronous and thus the reverse inequality holds.\\

 \item If $f\left( s \right)=\exp\left(s\right)=g\left( s \right)$,
 then $f$ and $g$ are $t^r$-synchronous for all $r \in
 \mathbb{R}$, so that we have
 \begin{multline*}
 \left\langle {\phi\left(B^{2r}\right)y,y} \right\rangle \left\langle {\varphi\left(\exp\left( 2A \right)\right)x,x} \right\rangle  + \left\langle {\varphi\left(A^{2r}\right)x,x} \right\rangle \left\langle {\phi\left(\exp\left( 2B \right)\right)y,y} \right\rangle  \\
  \ge   2 \left\langle {\varphi\left(A^{r}\exp\left( A \right)\right)x,x} \right\rangle\left\langle {\phi\left(B^{r}\exp\left(B \right)\right)y,y} \right\rangle.\\
\end{multline*}
 \end{enumerate}
Therefore, by choosing an appropriate function $h$ such that the
assumptions in Remark \ref{remark1} are fulfilled then one may
generate family of inequalities from \eqref{eq2.1}.
\end{remark}

\begin{corollary}
\label{cor3} Let $A$ be a selfadjoint operator with
$\spe\left(A\right)\subset \left[\gamma,\Gamma\right]$   for some
real numbers $\gamma,\Gamma$ with $0<\gamma<\Gamma$. Let $\varphi:\mathcal{B}\left(\mathcal{H} \right)\to \mathcal{B}\left(\mathcal{K} \right)$ be a linear unital map.  If $f: \left[
{\gamma,\Gamma} \right]\to \mathbb{R}$ is continuous and $f$ is
$t$-synchronous   on $\left[ {\gamma,\Gamma} \right]$, then
\begin{align}
\left\langle { \varphi \left(  A^2\right)  x,x } \right\rangle  \cdot\left\langle { \varphi \left(f^2\left( A \right) \right)x,x } \right\rangle 
\ge    \left\langle {\varphi \left(Af\left( A
	\right)\right)x,x  } \right\rangle^2 \label{eq2.7}
\end{align}
for each $x \in H$ with $\|x\|=1$. 
\end{corollary}
\begin{proof}
Setting $f=g$,  $\phi=\varphi$, $B=A$ and $y=x $ in Corollary \ref{cor2} we get the desired result.
\end{proof}

\begin{corollary}
\label{cor4} Let $A$ be an invertible selfadjoint operator with
$\spe\left(A\right)\subset \left[\gamma,\Gamma\right]$  for some
real numbers $\gamma,\Gamma$ with $\gamma<\Gamma$. Let $\phi,\varphi:\mathcal{B}\left(\mathcal{H} \right)\to \mathcal{B}\left(\mathcal{K} \right)$ be a linear unital maps. Let $h: \left[
{\gamma,\Gamma} \right]\to \mathbb{R}$ be a non-negative
continuous. If $f: \left[ {\gamma,\Gamma} \right]\to \mathbb{R}$
is continuous and  $h$-synchronous, then
 \begin{multline}
\label{eq2.9}\left\langle { \phi \left(h^2 \left( B \right)\right)  y,y } \right\rangle  \cdot\left\langle { \varphi \left(f\left( A \right) \right)x,x } \right\rangle
+  \left\langle { \varphi \left(h^2 \left( A \right)\right)x,x } \right\rangle \cdot  \left\langle { \phi \left(f\left( B \right) \right) 
	y,y } \right\rangle   
\\
\ge   \left\langle {\varphi \left(h\left( A \right)f\left( A
	\right)\right)x,x  } \right\rangle \cdot     \left\langle {\phi \left(h\left( B \right) \right) y,y } \right\rangle 
\\
+  \left\langle { \varphi \left( h\left( A
	\right) \right)x,x } \right\rangle\cdot  \left\langle { \phi \left( h\left( B \right)f\left( B
	\right)\right)y,y } \right\rangle   
\end{multline}
for each $x \in H$ with $\|x\|=1$. In particular, we have
 \begin{multline}
 \left\langle { \phi \left(h^2 \left( A^{-1} \right)\right)  x,x } \right\rangle  \cdot\left\langle { \varphi \left(f\left( A \right) \right)x,x } \right\rangle
+  \left\langle { \varphi \left(h^2 \left( A \right)\right)x,x } \right\rangle \cdot  \left\langle { \phi \left(f\left( A^{-1} \right) \right) 
	x,x } \right\rangle   
\\
\ge   \left\langle {\varphi \left(h\left( A \right)f\left( A
	\right)\right)x,x  } \right\rangle \cdot     \left\langle {\phi \left(h\left( A^{-1} \right) \right) x,x } \right\rangle\\ +  \left\langle { \varphi \left( h\left( A
	\right) \right)x,x } \right\rangle\cdot  \left\langle { \phi \left( h\left( A^{-1} \right)f\left( A^{-1}
	\right)\right)x,x } \right\rangle   \label{2.9*}
\end{multline}
 \end{corollary}

\begin{proof}
Sٍetting $g=1$ in \eqref{eq2.1} we get the first inequality \eqref{eq2.9}. The second inequality holds by setting $B=A^{-1}$ and $y=x$ in \eqref{eq2.9}.
\end{proof}

\begin{theorem}
\label{thm2.2} Let $A$ be a selfadjoint operator with
$\spe\left(A\right)\subset \left[\gamma,\Gamma\right]$ for some
real numbers $\gamma,\Gamma$ with $\gamma<\Gamma$. Let $\phi,\varphi:\mathcal{B}\left(\mathcal{H} \right)\to \mathcal{B}\left(\mathcal{K} \right)$ be two linear unital maps. Let $h: \left[
{\gamma,\Gamma} \right]\to \mathbb{R}$ be a non-negative
continuous. 
 If $f,g: \left[ {\gamma,\Gamma} \right]\to \mathbb{R}$ are
continuous and both $f$ and $g$ are $h$-synchronous
($h$-asynchronous) on $\left[ {\gamma,\Gamma} \right]$, then
\begin{multline}
\left\langle {\phi\left( h^2\left( B \right)\right) y,y} \right\rangle 
\cdot f\left( \left\langle {\varphi \left(A\right)x,x}
\right\rangle  \right)  g\left( \left\langle {\varphi \left(A\right)x,x}
\right\rangle \right)  
\\
+h^2\left( \left\langle {\varphi \left(A\right)x,x}
\right\rangle \right) \cdot \left\langle { \phi\left( f\left( B
	\right)g\left( B \right)\right)y,y } \right\rangle
\\
\ge (\le) \left\langle {\phi\left(h\left( B \right)g\left( B \right) \right)y,y} \right\rangle  f\left( \left\langle {\varphi \left(A\right)x,x}
\right\rangle  \right)h\left(  \left\langle {\varphi \left(A\right)x,x}
\right\rangle 
\right)
\\
+\left\langle { \phi \left( f\left( B
	\right)  h\left( B \right)\right) y,y} \right\rangle h\left( \left\langle {\varphi \left(A\right)x,x}
\right\rangle \right)g\left( \left\langle {\varphi \left(A\right)x,x}
\right\rangle \right) \label{eq2.10}
\end{multline}
 for any $x\in K$ with $\|x\|=\|y\|=1$.
\end{theorem}

\begin{proof}
Since $\gamma 1_{H}\le  \left\langle {Ax,x}
\right\rangle \le \Gamma 1_{H}$ then by employing $\varphi$, we get $\gamma 1_{K}\le   \varphi \left(A\right) 
  \le \Gamma 1_{K}$. So that $\gamma \le  \left\langle {\varphi \left(A\right)x,x}
\right\rangle \le \Gamma$  for any $x\in K$ with $\|x\|=1$.  Since $f, g$ are synchronous
\begin{multline}
\left[({h\left( \left\langle {\varphi \left(A\right)x,x}
	\right\rangle \right)f\left( t
\right) - h\left( t \right)f\left( \left\langle {\varphi \left(A\right)x,x}
\right\rangle  \right)} \right]
\\
\times\left[  {h\left(  \left\langle {\varphi \left(A\right)x,x}
	\right\rangle 
\right)g\left( t \right) - h\left( t \right)g\left( \left\langle {\varphi \left(A\right)x,x}
\right\rangle \right)} \right]\ge 0  \label{eq2.12}
\end{multline}
for any $t \in \left[\gamma,\Gamma\right]$ for any $x\in K$ with $\|x\|=1$.

Simplyfying the terms we have
 \begin{multline}
 h^2\left( t \right)f\left( \left\langle {\varphi \left(A\right)x,x}
 \right\rangle  \right)  g\left( \left\langle {\varphi \left(A\right)x,x}
 \right\rangle \right)  
\\
+h^2\left( \left\langle {\varphi \left(A\right)x,x}
\right\rangle \right)  \cdot f\left( t
\right)g\left( t \right)
 	\\
\ge h\left( t \right)g\left( t \right)f\left( \left\langle {\varphi \left(A\right)x,x}
 \right\rangle  \right)h\left(  \left\langle {\varphi \left(A\right)x,x}
 \right\rangle 
 \right)
 \\
+f\left( t
\right)  h\left( t \right)h\left( \left\langle {\varphi \left(A\right)x,x}
 \right\rangle \right)g\left( \left\langle {\varphi \left(A\right)x,x}
 \right\rangle \right).	
 	 \label{eq2.13}
 \end{multline} 
 Fix $x \in K$ with $\|x\| = 1$. By utilizing the continuous functional calculus
 for the operator $B$ we have by
the property \eqref{eq1.2} for inequality \eqref{eq2.13} we
have
 \begin{multline}
h^2\left( B \right)f\left( \left\langle {\varphi \left(A\right)x,x}
\right\rangle  \right)  g\left( \left\langle {\varphi \left(A\right)x,x}
\right\rangle \right)  
\\
+h^2\left( \left\langle {\varphi \left(A\right)x,x}
\right\rangle \right)   \cdot f\left( B
\right)g\left( B \right)
\\
\ge h\left( B \right)g\left( B \right)f\left( \left\langle {\varphi \left(A\right)x,x}
\right\rangle  \right)h\left(  \left\langle {\varphi \left(A\right)x,x}
\right\rangle 
\right)
\\
+f\left( B
\right)  h\left( B \right)h\left( \left\langle {\varphi \left(A\right)x,x}
\right\rangle \right)g\left( \left\langle {\varphi \left(A\right)x,x}
\right\rangle \right).	
\label{eq2.14}
\end{multline}
Taking the map $\phi$	 in the inequality \eqref{eq2.14}, we get
  \begin{multline}
\phi\left( h^2\left( B \right)\right) f\left( \left\langle {\varphi \left(A\right)x,x}
 \right\rangle  \right)  g\left( \left\langle {\varphi \left(A\right)x,x}
 \right\rangle \right)  
 \\
+h^2\left( \left\langle {\varphi \left(A\right)x,x}
\right\rangle \right)  \cdot \phi\left( f\left( B
 \right)g\left( B \right)\right) 
 \\
 \ge \phi\left(h\left( B \right)g\left( B \right) \right)f\left( \left\langle {\varphi \left(A\right)x,x}
 \right\rangle  \right)h\left(  \left\langle {\varphi \left(A\right)x,x}
 \right\rangle 
 \right)
 \\
 +\phi \left( f\left( B
 \right)  h\left( B \right)\right) h\left( \left\langle {\varphi \left(A\right)x,x}
 \right\rangle \right)g\left( \left\langle {\varphi \left(A\right)x,x}
 \right\rangle \right).	
 \label{eq2.15}
 \end{multline}
for any  bounded linear operator $B$ with $\spe\left({B}\right)
\subseteq \left[\gamma,\Gamma\right]$ and $y\in H$ with $\|y\|=1$. 

So that
we can write \eqref{eq2.15} in the form 
 \begin{multline*}
\left\langle {\phi\left( h^2\left( B \right)\right) y,y} \right\rangle 
f\left( \left\langle {\varphi \left(A\right)x,x}
\right\rangle  \right)  g\left( \left\langle {\varphi \left(A\right)x,x}
\right\rangle \right)  
\\
+h^2\left( \left\langle {\varphi \left(A\right)x,x}
\right\rangle \right)  \cdot \left\langle { \phi\left( f\left( B
	\right)g\left( B \right)\right)y,y } \right\rangle
\\
\ge\left\langle {\phi\left(h\left( B \right)g\left( B \right) \right)y,y} \right\rangle  f\left( \left\langle {\varphi \left(A\right)x,x}
\right\rangle  \right)h\left(  \left\langle {\varphi \left(A\right)x,x}
\right\rangle 
\right)
\\
+\left\langle { \phi \left( f\left( B
	\right)  h\left( B \right)\right) y,y} \right\rangle h\left( \left\langle {\varphi \left(A\right)x,x}
\right\rangle \right)g\left( \left\langle {\varphi \left(A\right)x,x}
\right\rangle \right).	
\end{multline*}
for each $x,y\in K$ with $\|x\|=\|y\|=1$, which proves the inequality in    \eqref{eq2.10}.
The reverse sense follows similarly, and the proof is completed. 
 \end{proof}
 \begin{remark}
 Taking $\phi=\varphi$ in \eqref{eq2.12} we get	
 \begin{multline*}
\left\langle {\varphi\left( h^2\left( B \right)\right) y,y} \right\rangle \cdot 
f\left( \left\langle {\varphi \left(A\right)x,x}
\right\rangle  \right)  g\left( \left\langle {\varphi \left(A\right)x,x}
\right\rangle \right)  
\\
+h^2\left( \left\langle {\varphi \left(A\right)x,x}
\right\rangle \right) \cdot \left\langle { \varphi\left( f\left( B
	\right)g\left( B \right)\right)y,y } \right\rangle\cdot  
\\
\ge (\le)\left\langle {\varphi\left(h\left( B \right)g\left( B \right) \right)y,y} \right\rangle  f\left( \left\langle {\varphi \left(A\right)x,x}
\right\rangle  \right)h\left(  \left\langle {\varphi \left(A\right)x,x}
\right\rangle 
\right)
\\
+\left\langle { \varphi \left( f\left( B
	\right)  h\left( B \right)\right) y,y} \right\rangle h\left( \left\langle {\varphi \left(A\right)x,x}
\right\rangle \right)g\left( \left\langle {\varphi \left(A\right)x,x}
\right\rangle \right). 
\end{multline*}
Also, by setting $B=A$ in \eqref{eq2.12} we get
 \begin{multline*}
\left\langle {\phi\left( h^2\left( A \right)\right) y,y} \right\rangle \cdot 
f\left( \left\langle {\varphi \left(A\right)x,x}
\right\rangle  \right)  g\left( \left\langle {\varphi \left(A\right)x,x}
\right\rangle \right)  
\\
+ h^2\left( \left\langle {\varphi \left(A\right)x,x}
\right\rangle \right)  \cdot \left\langle { \phi\left( f\left( A
	\right)g\left( A \right)\right)y,y } \right\rangle
\\
\ge (\le) \left\langle {\phi\left(h\left( A \right)g\left( A \right) \right)y,y} \right\rangle  f\left( \left\langle {\varphi \left(A\right)x,x}
\right\rangle  \right)h\left(  \left\langle {\varphi \left(A\right)x,x}
\right\rangle 
\right)
\\
+\left\langle { \phi \left( f\left( A
	\right)  h\left( A \right)\right) y,y} \right\rangle h\left( \left\langle {\varphi \left(A\right)x,x}
\right\rangle \right)g\left( \left\langle {\varphi \left(A\right)x,x}
\right\rangle \right). 
\end{multline*}
 \end{remark}

 \begin{corollary}
\label{cor5}Let $A$ be a selfadjoint operator with
$\spe\left(A\right)\subset \left[\gamma,\Gamma\right]$  for some
real numbers $\gamma,\Gamma$ with $\gamma<\Gamma$. Let $\phi,\varphi:\mathcal{B}\left(\mathcal{H} \right)\to \mathcal{B}\left(\mathcal{K} \right)$ be two linear unital maps. Let $h: \left[
{\gamma,\Gamma} \right]\to \mathbb{R}$ be a non-negative
continuous. If $f: \left[ {\gamma,\Gamma} \right]\to \mathbb{R}$
is continuous and  $h$-synchronous on $\left[ {\gamma,\Gamma}
\right]$, then
\begin{multline}
\left\langle {\phi\left( h^2\left( B \right)\right) y,y} \right\rangle 
\cdot f^2\left( \left\langle {\varphi \left(A\right)x,x}
\right\rangle  \right)    
+\left\langle { \phi\left( f^2\left( B
	\right) \right)y,y } \right\rangle\cdot h^2\left( \left\langle {\varphi \left(A\right)x,x}
\right\rangle \right)  
\\
\ge (\le)2 \left\langle {\phi\left(h\left( B \right)f\left( B \right) \right)y,y} \right\rangle  f\left( \left\langle {\varphi \left(A\right)x,x}
\right\rangle  \right)h\left(  \left\langle {\varphi \left(A\right)x,x}
\right\rangle 
\right)
  \label{eq2.17}
\end{multline}
 for any $x\in K$ with $\|x\|=\|y\|=1$. In particular, we have
 \begin{multline*}
 \left\langle {\varphi\left( h^2\left( B \right)\right) y,y} \right\rangle 
 \cdot f^2\left( \left\langle {\varphi \left(A\right)x,x}
 \right\rangle  \right)    
 +\left\langle { \varphi\left( f^2\left( B
 	\right) \right)y,y } \right\rangle\cdot h^2\left( \left\langle {\varphi \left(A\right)x,x}
 \right\rangle \right)  
 \\
 \ge (\le)2 \left\langle {\varphi\left(h\left( B \right)f\left( B \right) \right)y,y} \right\rangle  f\left( \left\langle {\varphi \left(A\right)x,x}
 \right\rangle  \right)h\left(  \left\langle {\varphi \left(A\right)x,x}
 \right\rangle 
 \right),
 \end{multline*}
   also, we have
\begin{multline*}
\left\langle {\phi\left( B^2\right) y,y} \right\rangle
\cdot f^2\left( \left\langle {\varphi \left(A\right)x,x}
\right\rangle  \right)  
+   \left\langle {\varphi \left(A^2\right)x,x}
\right\rangle  \cdot \left\langle { \phi\left( f^2\left( B
	\right) \right)y,y } \right\rangle
\\
\ge (\le) \left\langle {\phi\left(Bf\left( B \right) \right)y,y} \right\rangle  f\left( \left\langle {\varphi \left(A\right)x,x}
\right\rangle  \right)   \left\langle {\varphi \left(A\right)x,x}
\right\rangle^2 
\\
+\left\langle { \phi \left( Bf\left( B
	\right)   \right) y,y} \right\rangle   \left\langle {\varphi \left(A\right)x,x}
\right\rangle^2 f\left( \left\langle {\varphi \left(A\right)x,x}
\right\rangle \right) 
\end{multline*}
 for any $x\in K$ with $\|x\|=\|y\|=1$.  
\end{corollary}
\begin{proof}
Setting $f=g$ in \eqref{eq2.10}, respectively, we get the required
results.
\end{proof}

 \begin{corollary}
\label{cor6} Let $A$ be a selfadjoint operator with
$\spe\left(A\right)\subset \left[\gamma,\Gamma\right]$ for some
real numbers $\gamma,\Gamma$ with $0<\gamma<\Gamma$. Let $\phi,\varphi:\mathcal{B}\left(\mathcal{H} \right)\to \mathcal{B}\left(\mathcal{K} \right)$ be two linear unital maps.
 If $f : \left[ {\gamma,\Gamma} \right]\to \mathbb{R}$ are
continuous and $t$-synchronous on $\left[ {\gamma,\Gamma}
\right]$, then
\begin{multline}
\left\langle {\phi\left(B^2\right) y,y} \right\rangle 
\cdot f^2\left( \left\langle {\varphi \left(A\right)x,x}
\right\rangle  \right)    
+\left\langle { \phi\left( f^2\left( B
	\right) \right)y,y } \right\rangle\cdot \left\langle {\varphi \left(A\right)x,x}
\right\rangle^2 
\\
\ge (\le)2 \left\langle {\phi\left(Bf\left( B \right) \right)y,y} \right\rangle  f\left( \left\langle {\varphi \left(A\right)x,x}
\right\rangle  \right)   \left\langle {\varphi \left(A\right)x,x}
\right\rangle 
\end{multline}
 for any $x\in H$ with $\|x\|=1$.
 \end{corollary}
\begin{proof}
Setting $h(t)=t$ in \eqref{eq2.17}, respectively, we get the
required results.
\end{proof}

\begin{theorem}
\label{thm2.2-3} Let $A$ be a selfadjoint operator with
$\spe\left(A\right)\subset \left[\gamma,\Gamma\right]$ for some
real numbers $\gamma,\Gamma$ with $\gamma<\Gamma$. Let $\phi,\varphi:\mathcal{B}\left(\mathcal{H} \right)\to \mathcal{B}\left(\mathcal{K} \right)$ be two linear unital maps. Let $h: \left[
{\gamma,\Gamma} \right]\to \mathbb{R}_+$ be a positive function
on $\left[
{\gamma,\Gamma} \right]$.
 If $f,g: \left[ {\gamma,\Gamma} \right]\to \mathbb{R}_+$ are both positve, convex
  and  $h$-synchronous
  on $\left[ {\gamma,\Gamma} \right]$, then
\begin{multline}
h^2\left( \left\langle {\varphi(A)x,x} \right\rangle \right)
\left\langle {\phi(f\left(B\right))y,y} \right\rangle  \cdot  
\left\langle {\phi(g\left(B\right))y,y} \right\rangle   + h^2\left(\left\langle
{\phi(B)y,y} \right\rangle \right) \left\langle {\varphi(f\left(A\right))x,x}
\right\rangle \cdot   \left\langle {\varphi(g\left(A\right))x,x}
\right\rangle 
\\
\ge h\left( \left\langle {\varphi(A)x,x} \right\rangle \right)h\left(
\left\langle {\phi(B)y,y} \right\rangle \right) \left[f\left(
\left\langle {\phi({B})y,y} \right\rangle \right)  g\left( \left\langle
{\varphi(A)x,x} \right\rangle \right) + f\left( \left\langle {\varphi(A)x,x}
\right\rangle \right)  g\left( \left\langle {\phi(B)y,y} \right\rangle
\right)\right]
\label{eq2.10-3}
\end{multline}
for any $x,y\in H$ with $\|x\|=\|y\|=1$.
\end{theorem}

\begin{proof}
Since $f, g$ are $h$-synchronous and $\gamma \le  \left\langle {Ax,x}
\right\rangle \le \Gamma$, $\gamma \le  \left\langle {By,y}
\right\rangle \le \Gamma$ for any $x,y\in H$ with $\|x\|=\|y\|=1$,
we have
\begin{multline}
\left( {h\left( \left\langle {\varphi(A)x,x} \right\rangle \right)f\left(
\left\langle {\phi(A)y,y} \right\rangle \right) - h\left(\left\langle
{\phi(A)y,y} \right\rangle \right)f\left( \left\langle {\varphi(A)x,x}
\right\rangle \right)} \right)
\\
\times\left( {h\left( \left\langle {\varphi(A)x,x} \right\rangle
\right)g\left( \left\langle {\phi(A)y,y} \right\rangle  \right) -
h\left( \left\langle {\phi(A)y,y} \right\rangle  \right)g\left(
\left\langle {\varphi(A)x,x} \right\rangle \right)} \right) \ge 0
\label{eq2.12-3}
\end{multline}
for any $t \in \left[a,b\right]$ for any $x\in H$ with $\|x\|=1$.

Employing property \eqref{eq1.2} for inequality \eqref{eq2.12-3}
we have
\begin{multline}
 h^2\left( \left\langle {\varphi(A)x,x} \right\rangle \right)f\left(
\left\langle {\phi(B)y,y} \right\rangle \right)   g\left( \left\langle
{\phi(B)y,y} \right\rangle  \right)
\\
+  h^2\left(\left\langle {\varphi(B)y,y} \right\rangle \right)f\left(
\left\langle {\varphi(A)x,x} \right\rangle \right)  g\left( \left\langle
{\varphi(A)x,x} \right\rangle \right)
\\
- h\left( \left\langle {\varphi(A)x,x} \right\rangle \right)h\left(
\left\langle {\phi(B)y,y} \right\rangle \right)f\left( \left\langle
{\phi(B)y,y} \right\rangle \right)  g\left( \left\langle {\varphi(A)x,x}
\right\rangle \right)
\\
- h\left(\left\langle {\phi(B)y,y} \right\rangle \right) h\left(
\left\langle {\varphi(A)x,x} \right\rangle \right)f\left( \left\langle
{\varphi(A)x,x} \right\rangle \right)  g\left( \left\langle {\phi(B)y,y}
\right\rangle  \right) \ge 0 \label{eq2.13-3}
\end{multline}
for any  bounded linear operator $B$ with $\spe\left({B}\right)
\subseteq \left[\gamma,\Gamma\right]$ and $y\in H$ with $\|y\|=1$.

Now, since $f$ and $g$ are convex then we have
\begin{multline}
h^2\left( \left\langle {\varphi(A)x,x} \right\rangle \right)
\left\langle {\phi(f\left(B\right))y,y} \right\rangle  \cdot  
\left\langle {\phi(g\left(B\right))y,y} \right\rangle   + h^2\left(\left\langle
{\phi(B)y,y} \right\rangle \right) \left\langle {\varphi(f\left(A\right))x,x}
\right\rangle \cdot   \left\langle {\varphi(g\left(A\right))x,x}
\right\rangle 
\\
\ge h^2\left( \left\langle {\varphi(A)x,x} \right\rangle \right)f\left(
\left\langle {\phi(B)y,y} \right\rangle \right) \cdot  g\left(
\left\langle {\phi(B)y,y} \right\rangle  \right) + h^2\left(\left\langle
{\phi(B)y,y} \right\rangle \right)f\left( \left\langle {\varphi(A)x,x}
\right\rangle \right)\cdot  g\left( \left\langle {\varphi(A)x,x}
\right\rangle \right)
\\
\ge h\left( \left\langle {\varphi(A)x,x} \right\rangle \right)h\left(
\left\langle {\phi(B)y,y} \right\rangle \right) \left[f\left(
\left\langle {\phi({B})y,y} \right\rangle \right)  g\left( \left\langle
{\varphi(A)x,x} \right\rangle \right) + f\left( \left\langle {\varphi(A)x,x}
\right\rangle \right)  g\left( \left\langle {\phi(B)y,y} \right\rangle
\right)\right]  
\label{eq2.16-3}
\end{multline}
for each $x,y\in H$ with $\|x\|=\|y\|=1$. Setting $B=A^{-1}$ and
$y=x$ in \eqref{eq2.16-3} we get the required result in
\eqref{eq2.10-3}. The reverse sense follows similarly.
 \end{proof}


\begin{theorem}
	\label{thm2.2-3*} Let $A$ be a selfadjoint operator with
	$\spe\left(A\right)\subset \left[\gamma,\Gamma\right]$ for some
	real numbers $\gamma,\Gamma$ with $\gamma<\Gamma$. Let $\phi,\varphi:\mathcal{B}\left(\mathcal{H} \right)\to \mathcal{B}\left(\mathcal{K} \right)$ be two linear unital maps. Let $h: \left[
	{\gamma,\Gamma} \right]\to \mathbb{R}_+$ be a positive function
	on $\left[
	{\gamma,\Gamma} \right]$.
	If $f,g: \left[ {\gamma,\Gamma} \right]\to \mathbb{R}_+$ are both positve, convex
	and  $h$-synchronous
	on $\left[ {\gamma,\Gamma} \right]$, then
		\begin{align}
	&h^2\left( \left\langle {\varphi\left(A\right)x,x} \right\rangle \right)
	\left\langle {\phi\left( f\left(B\right)\right)y,y} \right\rangle  \cdot  
	\left\langle {\phi\left(g\left(B\right) \right)y,y} \right\rangle  
	\nonumber\\
	&\qquad+ h^2\left(\left\langle
	{\phi\left( B\right)y,y} \right\rangle \right) \left\langle {\varphi\left(f\left(A\right)\right)x,x}
	\right\rangle \cdot   \left\langle {\varphi\left( g\left(A\right)\right)x,x}
	\right\rangle 	\nonumber
	\\
	&\ge h^2\left( \left\langle {\varphi\left( A\right)x,x} \right\rangle \right)f\left(
	\left\langle {\phi\left( B\right)y,y} \right\rangle \right) \cdot  g\left(
	\left\langle {\phi\left( B\right)y,y} \right\rangle  \right) 
	\label{eq2.10-3*}\\
	&\qquad+ h^2\left(\left\langle
	{\phi\left(B \right)y,y} \right\rangle \right)f\left( \left\langle {\varphi\left(A \right)x,x}
	\right\rangle \right)\cdot  g\left( \left\langle {\varphi\left( A\right)x,x}
	\right\rangle \right)
	\nonumber\\
	&\ge h\left( \left\langle {\varphi\left( A\right)x,x} \right\rangle \right)h\left(
	\left\langle {\phi\left(B \right)y,y} \right\rangle \right) \times \left[f\left(
	\left\langle {\phi\left( B\right)y,y} \right\rangle \right)  g\left( \left\langle
	{\varphi\left( A\right)x,x} \right\rangle \right)\right.
	\nonumber\\
	&\qquad\qquad\qquad \left.	+ f\left( \left\langle {\varphi\left( A\right)x,x}
	\right\rangle \right)  g\left( \left\langle {\phi\left( B\right)y,y} \right\rangle
	\right)\right]  \nonumber
	\end{align}
	for any $x,y\in H$ with $\|x\|=\|y\|=1$.
\end{theorem}

\begin{proof}
ٍSince $\gamma \cdot 1_{H} \le A,B\le \Gamma \cdot 1_H$ then $\gamma \cdot 1_{K} \le \varphi \left(A\right)\le \Gamma \cdot 1_K$ and $\gamma \cdot1_K\le \phi \left(B\right)\le \Gamma \cdot 1_K$. So that for any $x,y\in H$ with $\|x\|=\|y\|=1$,
	we have $\gamma\le  \left\langle {\varphi\left(A\right) x,x}
	\right\rangle \le \Gamma$ and $\gamma \le  \left\langle {\phi\left(B\right) y,y}
	\right\rangle \le \Gamma$ 
	\begin{multline}
	\left( {h\left( \left\langle {\varphi\left(A\right)x,x} \right\rangle \right)f\left(
		\left\langle {\phi\left(B\right)y,y} \right\rangle \right) - h\left(\left\langle
		{\phi\left(B\right)y,y} \right\rangle \right)f\left( \left\langle {\varphi\left(A\right)x,x}
		\right\rangle \right)} \right)
	\\
	\times\left( {h\left( \left\langle {\varphi\left(A\right)x,x} \right\rangle
		\right)g\left( \left\langle {\phi\left(B\right)y,y} \right\rangle  \right) -
		h\left( \left\langle {\phi\left(B\right)y,y} \right\rangle  \right)g\left(
		\left\langle {\varphi\left(A\right)x,x} \right\rangle \right)} \right) \ge 0
	\label{eq2.12-3*}
	\end{multline}
	for any $t \in \left[a,b\right]$ for any $x\in H$ with $\|x\|=1$.
	
	Employing property \eqref{eq1.2} for inequality \eqref{eq2.12-3*}
	we have
	\begin{multline}
	h^2\left( \left\langle {\varphi\left(A\right)x,x} \right\rangle \right)f\left(
	\left\langle {\phi\left(B\right)y,y} \right\rangle \right)   g\left( \left\langle
	{\phi\left(B\right)y,y} \right\rangle  \right)
	\\
	+  h^2\left(\left\langle {\phi\left(B\right)y,y} \right\rangle \right)f\left(
	\left\langle {\varphi\left(A\right)x,x} \right\rangle \right)  g\left( \left\langle
	{\varphi\left(A\right)x,x} \right\rangle \right)
	\\
	- h\left( \left\langle {\varphi\left(A\right)x,x} \right\rangle \right)h\left(
	\left\langle {\phi\left(B\right)y,y} \right\rangle \right)f\left( \left\langle
	{\phi\left(B\right)y,y} \right\rangle \right)  g\left( \left\langle {\varphi\left(A\right)x,x}
	\right\rangle \right)
	\\
	- h\left(\left\langle {\phi\left(B\right)y,y} \right\rangle \right) h\left(
	\left\langle {\varphi\left(A\right)x,x} \right\rangle \right)f\left( \left\langle
	{\varphi\left(A\right)x,x} \right\rangle \right)  g\left( \left\langle {\phi\left(B\right)y,y}
	\right\rangle  \right) \ge 0. \label{eq2.13-3*}
	\end{multline}
	Now, since $f$ and $g$ are postive convex functions then we have
	\begin{align*}
	&h^2\left( \left\langle {\varphi\left(A\right)x,x} \right\rangle \right)
	\left\langle {\phi\left( f\left(B\right)\right)y,y} \right\rangle  \cdot  
	\left\langle {\phi\left(g\left(B\right) \right)y,y} \right\rangle  
	\nonumber\\
	 &\qquad+ h^2\left(\left\langle
	{\phi\left( B\right)y,y} \right\rangle \right) \left\langle {\varphi\left(f\left(A\right)\right)x,x}
	\right\rangle \cdot   \left\langle {\varphi\left( g\left(A\right)\right)x,x}
	\right\rangle 	\nonumber
	\\
	&\ge h^2\left( \left\langle {\varphi\left( A\right)x,x} \right\rangle \right)f\left(
	\left\langle {\phi\left( B\right)y,y} \right\rangle \right) \cdot  g\left(
	\left\langle {\phi\left( B\right)y,y} \right\rangle  \right) 
\\
	&\qquad+ h^2\left(\left\langle
	{\phi\left(B \right)y,y} \right\rangle \right)f\left( \left\langle {\varphi\left(A \right)x,x}
	\right\rangle \right)\cdot  g\left( \left\langle {\varphi\left( A\right)x,x}
	\right\rangle \right)
	\nonumber\\
	&\ge h\left( \left\langle {\varphi\left( A\right)x,x} \right\rangle \right)h\left(
	\left\langle {\phi\left(B \right)y,y} \right\rangle \right) \times \left[f\left(
	\left\langle {\phi\left( B\right)y,y} \right\rangle \right)  g\left( \left\langle
	{\varphi\left( A\right)x,x} \right\rangle \right)\right.
	\nonumber\\
	&\qquad\qquad\qquad \left.	+ f\left( \left\langle {\varphi\left( A\right)x,x}
	\right\rangle \right)  g\left( \left\langle {\phi\left( B\right)y,y} \right\rangle
	\right)\right]  \nonumber
	\end{align*}
	for each $x,y\in H$ with $\|x\|=\|y\|=1$, which proves  the required result in
	\eqref{eq2.10-3*}. The reverse sense follows similarly.
\end{proof}



\bibliographystyle{amsplain}

\end{document}